\documentclass[11pt,letterpaper]{amsart}
\usepackage{amssymb}
\usepackage{color}
\usepackage[linkcolor=blue,citecolor=blue, colorlinks=true,backref=true,bookmarks=true]{hyperref}

\theoremstyle{plain}
\newtheorem{theorem}{Theorem}[section]
\newtheorem{corollary}[theorem]{Corollary}

\newtheorem{lemma}[theorem]{Lemma}
\newtheorem{proposition}[theorem]{Proposition}

\theoremstyle{definition}
\newtheorem{remark}[theorem]{Remark}

\def\R{\mathbb{R}}
\def\S{\mathbb{S}}
\def\H{\mathbb{H}}
\def\Z{\mathbb{Z}}
\def\C{\mathbb{C}}

\def\N{\mathbb{N}}

\def\H{\mathbb{H}}

\newcommand{\PP}{\mathcal{P}}

\numberwithin{equation}{section}

\title{Improved Sobolev inequalities on CR sphere}

\begin{document}
\author{Zetian Yan}
\address{Department of Mathematics \\ Penn State University \\ University Park \\ PA 16802 \\ USA}
\email{zxy5156@psu.edu}
\begin{abstract}
We establish improved CR Sobolev inequalities on $S^{2n+1}$ under the vanishing of higher order moments of the volume element. As a direct application,
we give a simpler proof of the existence and the classification of minimizers of the CR invariant Sobolev inequalities which avoids complicated computation in Frank and Lieb's proof. Our argument relies on nice commutator identities involving the CR intertwining operators on $S^{2n+1}$ and handles both the fractional and integral cases. In the same spirit, we derive the classical sharp Sobolev inequalities using commutator identities on $S^n$. 
\end{abstract}
\maketitle

\section{Introduction}
This paper continues the program initiated in the work by the author \cite{Yan22} to simplify Frank and Lieb's proof \cite{FL10} of the sharp Sobolev inequalities for the embeddings $S^{\gamma,2}(\H^n)\hookrightarrow L^{\frac{2n+2}{n+1-\gamma}}(\H^n)$, and it handles both the fractional and integral cases. Our first objective is to establish improved CR Sobolev inequalities on $S^{2n+1}$ under the vanishing of higher order moments of the area elements; see Theorem \ref{improve} for a precise statement. As a direct application, we give a new proof of the existence of minimizers of CR Sobolev inequalities. Our second objective is to derive commutator identities involving intertwining operators on $S^{2n+1}$; see Theorem \ref{commutator} for a precise statement. By combining this with the Frank--Lieb argument developed in \cite{Case19}, we give a simpler and direct classification of optimizers. In particular, our proof avoids complicated computations, fully addressing an open problem of Frank and Lieb \cite{FL10}.

Our derivation of Theorem \ref{improve} is motivated by recent work in \cite{CH22} and \cite{HW21}, where Aubin's Moser--Trudinger--Onofri inequality and Sobolev inequalities on $S^n$ are improved under the vanishing of higher order moments of the volume element, respectively. The key observation in \cite{HW21} is that, if it is not true, there exists a sequence of functions $\{F_i\}$ in $W^{1,p}(S^n)$ such that
\begin{displaymath}
\|\nabla F_i\|^p_p\leqslant \frac{1}{\alpha}, \quad \|F_i\|_{p^*}=1, 
\end{displaymath}
and $\{F_i\}$ converges to zero weakly in $L^p(S^n)$, where $p^*=\frac{np}{n-p}$ and $\alpha$ is the leading coefficient of the improved Sobolev inequality \cite{HW21}. Combining this with the concentration compactness principle \cite{Lions85}, this contradicts the choice of $\alpha$. We will apply this idea to prove Theorem \ref{improve} in Section 3. The crucial point of the proof is the concentration compactness principle Lemma \ref{ccp} on CR sphere. 

Let $\mathcal{A}_{2\gamma}^{\theta_c}$, $0<\gamma<n+1$, be the CR intertwining operators associated with the canonical contact form $\theta_c$ on $\H^n$. Folland \cite{Folland75} proved that for all $f\in S^{\gamma,2}(\H^n)$, $p=\frac{2Q}{Q-2\gamma}$, there exists a positive constant $C$ such that 
\begin{equation}\label{optimal}
    \left(\int_{\H^n}|f|^pdu\right)^{\frac{2}{p}}\leqslant C\int_{\H^n}\bar{f} \mathcal{A}_{2\gamma}^{\theta_c} (f) du,
\end{equation}
where $S^{\gamma,2}(\H^n)$ is the Folland--Stein space introduced in \cite{Ponge05}, and $du=\theta_c\wedge (d\theta_c)^n$ is the volume form associated with $\theta_c$ on $\H^n$. Let $C_{n,2\gamma}$ be the smallest real number so that (\ref{optimal}) is true for all $f\in S^{\gamma,2}(\H^n)$. Equivalently, via the Cayley transform $\mathcal{C}:\H^n\to S^{2n+1}\backslash (0,\cdots, -1)$ given by 
\begin{displaymath}
    \mathcal{C}(z,t)=\left(\frac{2z}{1+|z|^2-it}, \frac{1-|z|^2+it}{1+|z|^2-it}\right),
\end{displaymath}
we have for all $F\in S^{\gamma,2}(S^{2n+1})$,
\begin{equation}\label{sharpSo}
\left(\int_{S^{2n+1}}|F|^p d\xi\right)^{\frac{2}{p}}\leqslant C_{n,2\gamma}\int_{S^{2n+1}}\bar{F} \mathcal{A}_{2\gamma}^{\theta_0} (F) d\xi,
\end{equation}
where $d\xi=\theta_0\wedge (d\theta_0)^n$ is the volume form associated with the canonical contact form $\theta_0$ on $S^{2n+1}$.

The explicit value of $C_{n,2\gamma}$ is computed by Frank and Lieb \cite{FL10} by classifying the extremal functions of (\ref{sharpSo}); i.e., equality holds in (\ref{sharpSo}) if and only if
\begin{displaymath}
F(\eta)=\frac{C}{|1-\xi \cdot \bar{\eta}|^{\frac{Q-2\gamma}{2}}}, \quad \eta\in S^{2n+1},
\end{displaymath}
for some $C\in \C, \xi\in \C^{n+1}, |\xi|<1$ and 
\begin{equation}
    C_{n,2k}=\left(\frac{1}{4\pi}\right)^{\gamma}\frac{\Gamma^2(\frac{n+1-\gamma}{2})}{\Gamma^2(\frac{n+1+\gamma}{2})}.
\end{equation}
Among the results of this article is a new computation of the value of $C_{n,2\gamma}$ and the classification of its optimizers; see Theorem \ref{class} for more details.

Recall that a homogeneous polynomial on $\C^{n+1}$ with bidegree $(j,k)$ is a polynomial $g$ such that for any $\lambda\in \C$, $z\in \C^{n+1}$,
\begin{displaymath}
    g(\lambda z)=\lambda^j \bar{\lambda}^{k}g(z).
\end{displaymath}
For a pair of nonnegative integers $(j, k)$, we denote
\begin{center}
$\PP_{j,k}:=\{$all homogeneous polynomials on $\C^{n+1}$ with bidegree $(j,k)\}$,

$\widetilde{\PP}_{j,k}:=\bigcup\limits_{(\tilde{j},\tilde{k})}\PP_{\tilde{j},\tilde{k}}$, \quad for all $\tilde{j}\leqslant j$, $\tilde{k}\leqslant k$,

$\overline{\PP}_{j,k}:=\left\{g\in \widetilde{\PP}_{j,k}; \int_{S^{2n+1}}g d\xi=0 \right\}$.
\end{center}
For $m\in \N$ and $0\leqslant \theta\leqslant 1$, as in \cite{HW21} we define
\begin{center}
$\mathcal{M}^c_{j,k}(S^{2n+1}):=\{\nu$; $\nu$ is a probability measure on $S^{2n+1}$ supported on countably many points such that $\int_{S^{2n+1}}gd\nu=0$, for all $g\in \overline{\PP}_{j,k}\}$
\end{center}
and 
\begin{displaymath}
\begin{split}
    \Theta&({j,k};\theta,2n+1)\\
    &:=\inf\left\{\sum_{i}\nu^{\theta}_i; \nu\in \mathcal{M}^c_{j,k}(S^{2n+1}), \mathrm{supp}(\nu)=\{x_i\}\subset S^{2n+1}, \nu_i=\nu(x_i)\right\}.
    \end{split}
\end{displaymath}

Our first result gives improved CR Sobolev inequalities under the assumption that higher moments vanish.
\begin{theorem}\label{improve}
Denote $Q=2n+2$, $p=\frac{2Q}{Q-2\gamma}$, $\gamma\in (0, n+1)$. Then for any $\epsilon>0$, and $F\in S^{\gamma,2}(S^{2n+1})$ with
\begin{equation}\label{constraint}
    \int_{S^{2n+1}}g|F|^p d\xi=0
\end{equation}
for all $g\in \overline{\PP}_{j,k}$, we have
\begin{equation}\label{asi}
\begin{split}
   \left(\int_{S^{2n+1}}|F|^p d\xi\right)^{\frac{2}{p}}\leqslant& \left(\frac{{C}_{n,2\gamma}}{\Theta({j,k};\frac{Q-2\gamma}{Q},2n+1)}+\epsilon\right) \int_{S^{2n+1}}\bar{F} \mathcal{A}_{2\gamma}^{\theta_0} (F) d\xi\\
   &+C(\epsilon)\int_{S^{2n+1}}|F|^2 d\xi,
\end{split}
\end{equation}
where $C(\epsilon)$ is a constant depending on $\epsilon$.
\end{theorem}
When $j+k=1$, the condition $\int_{S^{2n+1}}gd\nu=0$ for all $g\in \overline{\PP}_{j,k}$ is equivalent to the balanced conditions
\begin{equation}\label{balanced}
    \int_{S^{2n+1}}\xi_i d\nu=0, \quad i=1,\cdots, 2n+2,
\end{equation}
where $\{\xi_i\}_{i=1}^{2n+2}$ are coordinate functions on $\R^{2n+2}$. Therefore, the exact value of $\Theta({j,k}; \theta,2n+1)$ is $2^{1-\theta}$, shown in \cite{HW21}. We have the following corollary for balanced functions on $S^{2n+1}$.
\begin{corollary}\label{impro}
Denote $Q=2n+2$, $p=\frac{2Q}{Q-2\gamma}$, $\gamma\in (0, n+1)$. Then for any $\epsilon>0$, and $F\in S^{\gamma,2}(S^{2n+1})$ with
\begin{equation}
    \int_{S^{2n+1}}\xi_i|F|^p d\xi=0, \quad i=1,\cdots, 2n+2,
\end{equation}
we have
\begin{equation}
\left(\int_{S^{2n+1}}|F|^p d\xi\right)^{\frac{2}{p}}\leqslant \left(\frac{{C}_{n,2\gamma}}{2^{\frac{\gamma}{n+1}}}+\epsilon\right) \int_{S^{2n+1}}\bar{F} \mathcal{A}_{2\gamma}^{\theta_0} (F) d\xi+C(\epsilon)\int_{S^{2n+1}}|F|^2 d\xi,
\end{equation}
where $C(\epsilon)$ is a constant depending on $\epsilon$.
\end{corollary}

For $0<\gamma<n+1$, let $(w,w')\in \R^2$ such that $w+w'+n+1=\gamma$ and $w-w'\in \Z$. Through this article, without further comment, we will assume implicitly that $w-w'\in \Z$. On $(S^{2n+1}, \theta_0)$, $\mathcal{A}_{w,w'}^{\theta_0}$ is the unique self-adjoint intertwining operator of order $2\gamma$ in the sense of (\ref{intertwining}); see Proposition \ref{generalclass} for more details. 
\begin{theorem}\label{commutator}
Let $(S^{2n+1}, T^{1,0}S^{2n+1}, \theta_0)$ be the CR unit sphere. For $0<\gamma<n+1$, let $(w,w')\in \R^2$ such that $w+w'+n+1=\gamma$. $\mathcal{A}_{w,w'}^{\theta_0}$ are self-adjoint intertwining operators characterized in Proposition \ref{generalclass}. Then 
\begin{equation}\label{commutator2}
    \sum_{j=1}^{n+1}\bar{z}_j \left[\mathcal{A}_{w,w'}^{\theta_0},z_j\right]=\gamma (\gamma-1-w') \mathcal{A}_{w-1,w'}^{\theta_0}.
\end{equation}
\end{theorem}
Theorem \ref{commutator} handles both fractional and integral cases and covers the result in \cite{Yan22}. The proof of Theorem \ref{commutator} relies on direct computation of the spectrum with respect to spherical harmonics, which is different from that in \cite{Yan22}. In the case $w=w'$, combining (\ref{commutator2}) with Corollary \ref{sharpSo}, we give a new proof of sharp CR Sobolev inequalities on $S^{2n+1}$.
\begin{theorem}\label{class}
Let $(S^{2n+1}, T^{1,0}S^{2n+1}, \theta_0)$ be the CR unit sphere with the volume element $d\xi=\theta_0\wedge (d\theta_0)^n$. Denote $Q=2n+2$ and $\gamma\in (0, n+1)$. Then $u$ is a positive minimizer of (\ref{sharpSo}) if and only if
\begin{displaymath}
u(\eta)=\frac{C}{|1-\xi \cdot \bar{\eta}|^{\frac{Q-2\gamma}{2}}}, \quad \eta\in S^{2n+1}
\end{displaymath}
for some $C\in \C, \xi\in \C^{n+1}, |\xi|<1$.
\end{theorem}
The key ideas in the proof of Theorem \ref{class} are the following. After using the CR automorphism group $\mathcal{A}{\textbf{ut}}(S^{2n+1})$, we may assume that any minimizing sequence $\{F_i\}$ in $S^{\gamma,2}(S^{2n+1})$ is balanced \cite{FL10}. By combining Corollary \ref{impro} with a trick due to Lieb \cite{Lieb83}, we show that any balanced minimizing sequence $\{F_i\}$ has a subsequence which converges strongly to an optimizer $F\in S^{\gamma,2}(S^{2n+1})$ of (\ref{sharpSo}). Next, we follow the Frank--Lieb argument developed in \cite{Case19} to classify the optimizers. The Frank--Lieb argument involves three elements. First, the assumption of a local miminizer implies, via the second variation, a nice spectral estimate; see Proposition \ref{spectral} for a precise statement. Second, CR covariance implies that one can assume that the optimizer $F$ satisfies the balanced condition (\ref{balanced}), and in particular use first spherical harmonics as test functions in the previous spectral estimate. Third, by Theorem \ref{commutator}, one deduces that a balanced positive local minimizer is constant. 

Hang and Wang have announced a new proof of sharp CR Sobolev inequalities using a scheme of subcritical approximation \cite{HW22} which is quite different from the approach in this article.

Our methods also give a new proof of Beckner's sharp Sobolev inequalities on the sphere \cite{Beckner93} which is more direct than the argument of Frank and Lieb \cite{FL12a}.

\begin{theorem}\label{Beckner}
Let $\Delta_{S^n}$ be the Laplace--Beltrame operator on the standard sphere $(S^n, g_c)$. For $0<\gamma<\frac{n}{2}$, classical intertwining operators $P_{2\gamma}^{g_c}$ on $S^n$ are defined by
\begin{equation}\label{explicitform}
        P_{2\gamma}^{g_c}=\frac{\Gamma(B+\frac{1}{2}+\gamma)}{\Gamma(B+\frac{1}{2}-\gamma)}, \quad \quad  B=\sqrt{-\Delta_{S^n}+\left(\frac{n-1}{2}\right)^2}.
\end{equation}
Then for all $F\in W^{\gamma,2}(S^n)$,
\begin{equation}\label{classicalHLS}
    \frac{\Gamma(\frac{n+2\gamma}{2})}{\Gamma(\frac{n-2\gamma}{2})}\omega_n^{\frac{2\gamma}{n}}\left(\int_{S^n}|F|^{\frac{2n}{n-2\gamma}}d\sigma\right)^{\frac{n-2\gamma}{n}}\leqslant \int_{S^n}FP_{2\gamma}^{g_c}(F) d\sigma, 
\end{equation}
where $\omega_n$ is the volume of $S^n$ and $d\sigma$ is the Lebesgue measure on $S^n$. Equality holds if and only if
\begin{equation}
    F(\eta)=C|1-\langle\eta, \xi\rangle|^{\frac{2\gamma-n}{2}}, \quad \eta\in S^n
\end{equation}
for some $C\in \R$, $\xi\in \R^{n+1}$, $|\xi|<1$.
\end{theorem}
The main ingredient in our new proof of Theorem \ref{Beckner} is a commutator identity for the intertwining operators on $S^n$ which is already known when $\gamma\in \N$ \cite{Case19}.

\begin{theorem}\label{classicalcom}
    Let $(S^n, g_c)$ be the standard sphere with the canonical metric $g_c$ and $P_{2\gamma}^{g_c}$ be the intertwining operator of order $2\gamma$. Then
    \begin{equation}
\sum_{j=1}^{n+1}x_j\left[P_{2\gamma}^{g_c},x_j\right]=\gamma(n+2\gamma-2)P_{2(\gamma-1)}^{g_c},
    \end{equation}
where $x_1,\cdots, x_{n+1}$ are the standard coordinates on $\R^{n+1}$ and $P_{2(\gamma-1)}^{g_c}$ should be regarded as $\left(P_{2(1-\gamma)}^{g_c}\right)^{-1}$ for $0<\gamma<1$.
\end{theorem}

This article is organized as follows.

In Section 2, we review some basic concepts on CR sphere, spherical harmonics and CR intertwining operators. In Section 3, we prove Theorem \ref{improve} following the idea in \cite{HW21}. Section 4 deals with the commutator identity and spectral estimate needed for the Frank--Lieb argument. As a direct application, we give a
new proof of sharp CR Sobolev inequalities on $S^{2n+1}$. In Section 5, we provide the proof of Theorem \ref{classicalcom} and a new proof of classical sharp Sobolev inequalities (\ref{classicalHLS}).

{\bf{Acknowledgements.}} I would like to express my deep gratitude to my advisor, Prof. Jeffrey S. Case, for the patient guidance and useful critiques of this work. I also would like to thank Yu Gao, Nan Wu and Xingyu Zhu for useful discussion.

\section{Preliminaries}
\subsection{CR sphere and spherical harmonics}
Let $S^{2n+1}\subset \C^{n+1}$ be the unit sphere centered at the origin with the canonical holomorphic tangent space $T^{1,0}S^{2n+1}$ and the contact form $\theta_0$ given by
\begin{equation*}
\begin{split}
    &T^{1,0}S^{2n+1}=T^{1,0}\C^{n+1}\cap \C TS^{2n+1}, \\
    &\theta_0=\frac{i}{2}\sum_{i=1}^{n+1}\left(z^i d\bar{z}^i-\bar{z}^i dz^i\right)\big|_{S^{2n+1}}.
    \end{split}
\end{equation*}
On $(S^{2n+1}, T^{1,0}S^{2n+1}, \theta_0)$, the sublaplacian is defined as
\begin{equation*}
    \mathcal{L}=-\frac{1}{2}\sum_{j=1}^{n+1} \left(T_j\overline{T}_j+\overline{T}_jT_j\right),
\end{equation*}
where
\begin{equation*}
    T_j=\frac{\partial}{\partial z_j}-\bar{z}_j \sum_{k=1}^{n+1}z_k\frac{\partial}{\partial z_k}, \quad \quad \overline{T}_j=\frac{\partial}{\partial \bar{z}_j}-z_j \sum_{k=1}^{n+1}\bar{z}_k\frac{\partial}{\partial \bar{z}_k},
\end{equation*}
and the conformal sublaplacian is defined as 
\begin{equation*}
    \mathcal{D}=\mathcal{L}+\frac{n^2}{4}.
\end{equation*}

Let $\mathcal{P}$ denote the space of complex-valued polynomials on $\C^{n+1}$ and $\mathcal{H}$ be the subspace of harmonic polynomials. For all $h\in \N_0$, let $\mathcal{P}_h\subset \mathcal{P}$ denote the space of homogeneous polynomials of degree $h$ and set $\mathcal{H}_h=\mathcal{P}_h\cap \mathcal{H}$. Clearly, we have the following decomposition
\begin{equation*}
    \mathcal{P}=\bigoplus_{h\in \N}\mathcal{P}_h, \quad \quad \mathcal{H}=\bigoplus_{h\in \N}\mathcal{H}_h.
\end{equation*}
It is well-known that 
\begin{equation}
    \mathcal{P}_h=\mathcal{H}_h\oplus |z|^2\mathcal{P}_{h-2} \quad \quad ~~\mbox{for all}~~ h\in \N_0,
\end{equation}
where $z=(z_1,\cdots, z_{n+1})\in \C^{n+1}$ and $|z|^2=z\cdot \bar{z}$. Given $j,k\in \N$, we may then define $\PP_{j,k}$ to be the space of homogeneous polynomials on $\C^{n+1}$ with bidegree $(j,k)$, i.e., for $g\in \PP_{j,k}$, we have $Zg=jg$ and $\overline{Z}g=kg$, where $Z$ and $\overline{Z}$ are the holomorphic and antiholomorphic Euler fields given by
\begin{equation}
    Z=\sum_{l=1}^{n+1}z_l\frac{\partial}{\partial z_l}, \quad \quad \overline{Z}=\sum_{l=1}^{n+1}\bar{z}_l\frac{\partial}{\partial \bar{z}_l},
\end{equation}
respectively. Obviously, we have
\begin{equation*}
    \mathcal{P}_h=\bigoplus_{j+k=h} \mathcal{P}_{j,k}, \quad \quad \mathcal{H}_h=\bigoplus_{j+k=h} \mathcal{H}_{j,k},
\end{equation*}
where $\mathcal{H}_{j,k}=\mathcal{H}_{h}\cap \mathcal{P}_{j,k}$.

By the Stone--Weierstrass theorem, we know that the space $L^2(S^{2n+1})$ endowed with the inner product
\begin{equation*}
    (F,G)=\int_{S^{2n+1}}F\overline{G}d\xi, \quad \quad d\xi=\theta_0\wedge (d\theta_0)^n,
\end{equation*}
can be decomposed as 
\begin{equation*}
    L^2(S^{2n+1})=\overline{\bigoplus_{j,k\in\N_0}\mathcal{H}^{\S}_{j,k}},
\end{equation*} 
where $\mathcal{H}^{\S}_{j,k}$ is the restriction of $\mathcal{H}_{j,k}$ to the sphere. 
Through this paper, the superscript $\S$ will be used to denote the sets of restrictions to $S^{2n+1}$. The dimension of $\mathcal{H}^{\S}_{j,k}$ \cite{AC16} is 
\begin{equation*}
    \mathrm{dim} (\mathcal{H}^{\S}_{j,k})=m_{j,k}:=\frac{(j+n-1)!(k+n-1)!(j+k+n)!}{n!(n-1)!j!k!},
\end{equation*}
and if $\{Y_{j,k}^l\}_{l=1}^{m_{j,k}}$ is an orthonormal basis of $\mathcal{H}^{\S}_{j,k}$, then the zonal harmonics are defined as
\begin{equation*}
    \Phi_{j,k}(\zeta,\eta)=\sum_{l=1}^{m_{j,k}}Y_{j,k}^l(\zeta)\overline{Y_{j,k}^l(\eta)}.
\end{equation*}
It is known from \cite{BFM07} that
\begin{equation}\label{zonal}
    \Phi_{j,k}(\zeta,\eta)=\Phi_{j,k}(\bar{\zeta}\cdot\eta):=\frac{(k+n-1)!(j+k+n)}{\omega_{2n+1}n!k!}\left(\bar{\zeta}\cdot \eta\right)^{k-j}P_j^{(n-1,k-j)}(2|\bar{\zeta}\cdot \eta|-1)
\end{equation}
if $j\leqslant k$, and $\Phi_{j,k}(\zeta,\eta):=\overline{\Phi_{k,j}(\zeta,\eta)}$ if $k\leqslant j$, where $P_{n}^{(\alpha,\beta)}$ are the Jacobi polynomials.

\subsection{Folland--Stein spaces and intertwining operators on CR sphere} The Folland--Stein spaces on $S^{2n+1}$ can be defined in terms of the powers of the conformal subplacian; see \cite{ACDB04,ADB06,BFM07,Folland75} for more details. We summarize the main properties below. 

It is well-known that for $Y_{j,k}\in \mathcal{H}^{\S}_{j,k}$,
\begin{equation}\label{decom}
    \mathcal{D}Y_{j,k}=\lambda_j \lambda_{k}Y_{j,k}, \quad \lambda_j=j+\frac{n}{2}.
\end{equation}
For $F\in L^2(S^{2n+1})$, we can write
\begin{equation}\label{expansion}
    F=\sum_{j,k\in \N_0}\sum_{l=1}^{m_{j,k}} c_{j,k}^l(F)Y^l_{j,k}, \quad c_{j,k}^l(F)=\int_{S^{2n+1}}F\overline{Y}^l_{j,k}d\xi;
\end{equation}
in particular, for $F\in C^{\infty}(S^{2n+1})$ and any $\gamma\in \R_+$, (\ref{decom}) implies that
\begin{equation}\label{finite}
    \sum_{j,k\in \N_0}\sum_{l=1}^{m_{j,k}} \left(\lambda_j \lambda_{k}\right)^\gamma |c_{j,k}^l(F)|^2< \infty.
\end{equation}
For $F\in C^{\infty}(S^{2n+1})$ and any $\gamma\in \R_+$, we define
\begin{equation}\label{Ddefi}
    \mathcal{D}^{\frac{\gamma}{2}}F=\sum_{j,k\in \N_0}\sum_{l=1}^{m_{j,k}} \left(\lambda_j \lambda_{k}\right)^{\frac{\gamma}{2}}c_{j,k}^l(F)Y^l_{j,k}, 
\end{equation}
so that $\mathcal{D}^{\frac{\gamma}{2}}$ extends naturally to the space of distributions on the sphere. For $\gamma>0$, $p\geqslant 1$, we let
\begin{equation*}
    S^{\gamma,p}=\left\{F\in L^p: \mathcal{D}^{\frac{\gamma}{2}}F\in L^p\right\},
\end{equation*}
endowed with norm
\begin{equation*}
    \|F\|_{S^{\gamma,p}}=\|\mathcal{D}^{\frac{\gamma}{2}}F\|_p;
\end{equation*}
the space $S^{\gamma,p}$ is the completion of $C^{\infty}(S^{2n+1})$ this norm.

$S^{\gamma,2}$ is the space of $F$ in $L^2$ so that (\ref{finite}) holds. It is a Hilbert space with inner product and norm
\begin{equation*}
    (F,G)_{S^{\gamma,2}}=\int_{S^{2n+1}}\mathcal{D}^{\frac{\gamma}{2}}F \overline{\mathcal{D}^{\frac{\gamma}{2}}G}d\xi, \quad \|F\|_{S^{\gamma,2}}=(F,F)^{\frac{1}{2}}_{S^{\gamma,2}}.
\end{equation*}

The group $SU(n+1,1)$ acts as a group of CR automorphisms on $S^{2n+1}$, and therefore on $\H^n$ by means of Cayley transform. Recall that a CR automorphism is a diffeomorphism $\tau:S^{2n+1}\to S^{2n+1}$ that preserves the contact form; i.e., $\tau^* \theta_0=|J_{\tau}|^2\theta_0$, where the function $J_{\tau}$ is defined by 
\begin{equation*}
    \tau^{*}\left(dz_1\wedge\cdots \wedge dz_{n+1}\right)=J_{\tau}^{n+2}dz_1\wedge\cdots \wedge dz_{n+1}
\end{equation*}
such that $|J_{\tau}|^{2n+2}$ is the Jacobian determinant of $\tau$. Denote the CR automorphism group of $S^{2n+1}$ by $\mathcal{A}{\textbf{ut}}(S^{2n+1})$. The functions $|J_{\tau}|$ with $\tau\in \mathcal{A}{\textbf{ut}}(S^{2n+1})$ can be parametrized as
\begin{equation}
   |J_{\tau}(\zeta)|=\frac{C}{|1-\xi \cdot \bar{\zeta}|}, \quad \zeta\in S^{2n+1}, C\in \C, \xi\in \C^{n+1}, |\xi|<1.
\end{equation}

The conformal sublaplacian $\mathcal{D}$ is intertwining in the sense that for any $F\in C^{\infty}(S^{2n+1})$, $\tau\in \mathcal{A}{\textbf{ut}}(S^{2n+1})$,
\begin{equation*}
|J_{\tau}|^{\frac{Q+2}{2}}\left(\mathcal{D}F\right)\circ \tau=\mathcal{D}\left(|J_{\tau}|^{\frac{Q-2}{2}}\left(F\circ \tau\right)\right), \quad Q=2n+2.
\end{equation*}
For $0<\gamma<\frac{Q}{2}$, let $(w,w')\in \R^2$ such that $w+w'+n+1=\gamma$. The general intertwining operator $\mathcal{A}_{w,w'}^{\theta_0}$ of order $2\gamma$ is defined by the following property: for any $F\in C^{\infty}(S^{2n+1})$, $\tau\in \mathcal{A}{\textbf{ut}}(S^{2n+1})$,
\begin{equation}\label{intertwining}
    J_{\tau}^{\gamma-w}\bar{J}_{\tau}^{\gamma-w'}\left(\mathcal{A}_{w,w'}^{\theta_0}F\right)\circ \tau= \mathcal{A}_{w,w'}^{\theta_0}\left(J_{\tau}^{-w}\bar{J}_{\tau}^{-w'}\left(F\circ \tau\right)\right)
\end{equation}
In other words, the pullback of $\mathcal{A}_{w,w'}^{\theta_0}$ by a CR automorphism $\tau$ satisfies
\begin{equation*}
    \tau^* \mathcal{A}_{w,w'}^{\theta_0} \left(\tau^{-1}\right)^*=J_{\tau}^{-\gamma+w}\bar{J}_{\tau}^{-\gamma+w'}\mathcal{A}_{w,w'}^{\theta_0}J_{\tau}^{-w}\bar{J}_{\tau}^{-w'}, \quad \tau^*F=F\circ \tau.
\end{equation*}

In the same spirit as \cite{BFM07}, for the reader's sake in Appendix A we offer a seld-contained proof of the spectral characterization of intertwining operators. It is known from Proposition \ref{generalclass} that a self-adjoint operator satisfying (\ref{intertwining}) is diagonal with respect to the spherical harmonics, and its spectrum is completely determined up to a multiplicative constant by the functions
\begin{equation}\label{spectrum}
     \lambda_j(w)=\frac{\Gamma(j+\gamma-w)}{\Gamma(j-w)}
\end{equation}
in the sense that up to constant the spectrum is precisely $\{\lambda_j(w)\lambda_{k}(w')\}$. From now on we will choose such constant to be $1$; i.e., $\mathcal{A}_{w,w'}^{\theta_0}$ will be the operator on $C^{\infty}(S^{2n+1})$ such that
\begin{equation}\label{form}
\mathcal{A}_{w,w'}^{\theta_0}Y_{j,k}=\lambda_j(w)\lambda_{k}(w')Y_{j,k},\quad   Y_{j,k}\in \mathcal{H}^{\S}_{j,k}.
\end{equation}
Observe that when $w=w'$, $\mathcal{A}_{w,w'}^{\theta_0}=\mathcal{A}_{2\gamma}^{\theta_0}$ are intertwining operators characterized in \cite[Proposition A.1]{BFM07}. In particular, in the case $\gamma=1$, we have $\lambda_j(-\frac{n}{2})=j+\frac{n}{2}$, and we recover the conformal sublaplacian; i.e., $\mathcal{A}_{2}^{\theta_0}=\mathcal{D}$.


\section{Improved CR Sobolev inequalities}
In this section, we improve the CR Sobolev inequalities on $S^{2n+1}$ under the vanishing of higher order moments of the volume elements. We start from the concentration compactness principle on CR unit sphere. 

\begin{lemma}\label{ccp}
Let $(S^{2n+1}, T^{1,0}S^{2n+1}, \theta_0)$ be the CR unit sphere with the volume element $d\xi=\theta_0\wedge (d\theta_0)^n$. We assume that $F_i$ is a sequence of complex-valued functions bounded in $S^{\gamma,2}(S^{2n+1})$, $\gamma\in (0,n+1)$ and $p=\frac{2Q}{Q-2\gamma}$, such that $F_i\rightharpoonup F$ weakly in $S^{\gamma,2}(S^{2n+1})$. Moreover, we assume as measures, 
\begin{equation}
|F_i|^p d\xi \to |F|^pd\xi+\nu, \quad \bar{F_i}\mathcal{A}_{2\gamma}^{\theta_0}(F_i) d\xi \to \bar{F}\mathcal{A}_{2\gamma}^{\theta_0}(F) d\xi+\sigma.
\end{equation}
Then we can find countably many points $\{x_i\}\subset S^{2n+1}$ such that
\begin{displaymath}
    \nu=\sum_{i}\nu_i \delta_{x_i}, \quad \nu_i^{\frac{2}{p}}\leqslant C_{n,2\gamma} \sigma_i
\end{displaymath}
where $\nu_i=\nu(x_i)$ and $\sigma_i=\sigma(x_i)$.
\end{lemma}
\begin{proof}
The proof is totally similar to that of Lemma 3.1 in \cite{Yan22}. We leave it to readers.
\end{proof}

\begin{remark}
The dual version of Lemma \ref{ccp} on $\H^n$ was mentioned in \cite{Han13}.
\end{remark}

Now we can prove Theorem \ref{improve} with the help of Lemma \ref{ccp}. 
\begin{proof}[Proof of Theorem \ref{improve}]
Let 
\begin{displaymath}
    \alpha=\frac{{C}_{n,2\gamma}}{\Theta({j,k};\frac{Q-2\gamma}{Q},2n+1)}+\epsilon.
\end{displaymath}
If (\ref{asi}) is not true, then for any $l\in \N$, we can find a $F_l\in S^{\gamma,2}(S^{2n+1})$ such that 
\begin{equation}\label{constraint}
    \int_{S^{2n+1}}g|F|^p d\xi=0
\end{equation}
for all $g\in \overline{\PP}_{j,k}$, and 
\begin{equation}
    \left(\int_{S^{2n+1}}|F_l|^p d\xi\right)^{\frac{2}{p}}> \alpha \int_{S^{2n+1}}\bar{F}_l \mathcal{A}_{2\gamma}^{\theta_0}(F_l) d\xi+l\int_{S^{2n+1}}|F_l|^2 d\xi.
\end{equation}
We may assume
\begin{displaymath}
    \left(\int_{S^{2n+1}}|F_l|^p d\xi\right)^{\frac{2}{p}}=1.
\end{displaymath}
Then
\begin{equation}\label{small}
    \int_{S^{2n+1}}\bar{F}_l \mathcal{A}_{2\gamma}^{\theta_0}(F_l) d\xi<\frac{1}{\alpha}, \quad \int_{S^{2n+1}}|F_l|^2 d\xi<\frac{1}{l}.
\end{equation}
It follows that $F_l\rightharpoonup 0$ weakly in $S^{\gamma,2}(S^{2n+1})$. After passing to a subsequence we have
\begin{equation}
    \bar{F}_l \mathcal{A}_{2\gamma}^{\theta_0}(F_l) d\xi\to \sigma, \quad |F_l|^p d\xi \to \nu.
\end{equation}
By Lemma \ref{ccp} we can find countably many points $\{x_i\}\in S^{2n+1}$ such that
\begin{equation}\label{bubble}
    \nu=\sum_{i}\nu_i \delta_{x_i}, \quad \nu_i^{\frac{2}{p}}\leqslant {C}_{n,2\gamma} \sigma_i,
\end{equation}
where $\nu_i=\nu(x_i)$, $\sigma_i=\sigma(x_i)$. Then
\begin{equation}\label{volume}
    \nu(S^{2n+1})=1,\quad \sigma(S^{2n+1})<\frac{1}{\alpha}.
\end{equation}
It follows from (\ref{constraint}) and (\ref{small}) that $\int_{S^{2n+1}}gd\nu=0$ for all $g\in \overline{\PP}_{j,k}$, hence $\nu\in \mathcal{M}^c_{j,k}(S^{2n+1})$. By definition of $\Theta(j,k;\theta,2n+1)$, (\ref{bubble}) and (\ref{volume}) we have
\begin{displaymath}
    \Theta({j,k};\frac{Q-2\gamma}{Q},2n+1)\leqslant \sum_i \nu_i^{\frac{2}{p}}\leqslant \sum_i {C}_{n,2\gamma} \sigma_i={C}_{n,2\gamma} \sigma(S^{2n+1})\leqslant \frac{{C}_{n,2\gamma}}{\alpha}.
\end{displaymath}
Hence
\begin{displaymath}
    \alpha\leqslant \frac{{C}_{n,2\gamma}}{\Theta({j,k};\frac{Q-2\gamma}{Q},2n+1)}.
\end{displaymath}
This contradicts the choice of $\alpha$.
\end{proof}

\section{Sharp CR Sobolev inequalities}
The main result in this section is the classification of optimizers of the sharp Sobolev inequalities (\ref{sharpSo}). First, by improved Sobolev inequalities Corollary \ref{sharpSo}, we prove the existence of an optimizer.
\begin{proposition}
Denote $Q=2n+2$, $p=\frac{2Q}{Q-2\gamma}$. Then the sharp constant in (\ref{sharpSo}) is attained. Moreover, for any minimizing sequence $\{F_i\}$ there is a subsequence $\{F_{i_m}\}$ and a sequence $\{\Phi_{i_m}\}$ in the CR automorphism group $\mathcal{A}(S^{2n+1})$ of $S^{2n+1}$ such that
\begin{equation}\label{replaced}
    F^{\Phi}_{i_m}=|J_{\Phi_{i_m}}|^{\frac{1}{p}}\Phi^{*}_{i_m}F_{i_m}
\end{equation}
converges strongly in $S^{\gamma,2}(S^{2n+1})$, where $|J_{\Phi_{i_m}}|$ is the determinant of the Jacobian of $\Phi_{i_m}$.
\end{proposition}
\begin{proof}
By Lemma B.1 in \cite{FL10}, for each $F_i$, there exists an element $\Phi_{i}$ of $\mathcal{A}(S^{2n+1})$ such that $F^{\Phi}_{i}$ defined in (\ref{replaced}) satisfies the balanced conditions (\ref{balanced}). By CR invariance, we may replace $F_i$ by $F^{\Phi}_{i}$ in (\ref{sharpSo}). Thus, $\{F^{\Phi}_{i}\}$ is also a minimizing sequence. Passing to a subsequence $\{F_{i_m}\}$, we may assume that $F^{\Phi}_{i_m}\rightharpoonup F$ weakly in $S^{\gamma,2}(S^{2n+1})$ and that $F^{\Phi}_{i_m} \to F$ strongly in $L^2(S^{2n+1})$. 

Without loss of generality, we may assume $\int_{S^{2n+1}}|F^{\Phi}_{i_m}|^p d\xi=1$ for all $i_m$. By Corollary \ref{impro}, we can choose a sufficiently small $\epsilon$ such that as $i_m\to \infty$,
\begin{equation}\label{nonzero}
    \begin{split}
        1=\int_{S^{2n+1}}|F^{\Phi}_{i_m}|^p d\xi&\leqslant \left(\frac{{C}_{n,2\gamma}}{2^{\frac{\gamma}{n+1}}}+\epsilon\right)\int_{S^{2n+1}}\bar{F}^{\Phi}_{i_m} \mathcal{A}_{2\gamma}^{\theta_0}(F^{\Phi}_{i_m}) d\xi +C(\epsilon)\int_{S^{2n+1}}|F^{\Phi}_{i_m}|^2 d\xi,\\
        &\leqslant \left({C}_{n,2\gamma}-\epsilon\right)\int_{S^{2n+1}}\bar{F}^{\Phi}_{i_m} \mathcal{A}_{2\gamma}^{\theta_0}(F^{\Phi}_{i_m}) d\xi+C(\epsilon)\int_{S^{2n+1}}|F^{\Phi}_{i_m}|^2 d\xi,\\
        &\to \left(1-\frac{\epsilon}{{C}_{n,2\gamma}}\right)+C(\epsilon)\int_{S^{2n+1}}|F|^2 d\xi,
    \end{split}
\end{equation}
which implies $F\neq 0$.

By \cite[Lemma2.6]{Lieb83}, we have
\begin{equation}\label{dep}
        1=\|F^{\Phi}_{i_m}\|_p^p=\|F\|_p^p+\|F^{\Phi}_{i_m}-F\|_p^p+o(1),
\end{equation}
Since for $a,b,c\geqslant 0$ and $p> 2$, $\left(a^p+b^p+c^p\right)^{\frac{2}{p}}\leqslant a^2+b^2+c^2$, we have 
\begin{equation}\label{str}
    \begin{split}
        \left(\int_{S^{2n+1}}|F^{\Phi}_{i_m}|^p d\xi\right)^{\frac{2}{p}}&- \left(\int_{S^{2n+1}}|F|^p d\xi\right)^{\frac{2}{p}}\\
        &\leqslant \left(\int_{S^{2n+1}}|F^{\Phi}_{i_m}-F|^p d\xi\right)^{\frac{2}{p}}+o(1)\\
        &\leqslant {C}_{n,2\gamma}\int_{S^{2n+1}}\left(\bar{F}^{\Phi}_{i_m}-\bar{F}\right) \mathcal{A}_{2\gamma}^{\theta_0} \left(F^{\Phi}_{i_m}-F\right) d\xi +o(1)\\
        &\leqslant {C}_{n,2\gamma}\int_{S^{2n+1}}\left(\bar{F}^{\Phi}_{i_m} \mathcal{A}_{2\gamma}^{\theta_0}(F^{\Phi}_{i_m})-\bar{F} \mathcal{A}_{2\gamma}^{\theta_0}(F) \right)d\xi+o(1).
    \end{split}
\end{equation}
Since $\{F^{\Phi}_{i_m}\}$ is a minimizing sequence, as $i_m\to \infty$, we conclude that
\begin{equation}\label{infimum}
    1-\left(\int_{S^{2n+1}}|F|^p d\xi\right)^{\frac{2}{p}}\leqslant 1-{C}_{n,2\gamma}\int_{S^{2n+1}}\bar{F}\mathcal{A}_{2\gamma}^{\theta_0}(F) d\xi.
\end{equation}
This implies that $F$ is a minimizer because $F\neq 0$.

In order to see that the convergence of $\{F^{\Phi}_{i_m}\}$ in $L^p(S^{2n+1})$ is strong, we need to show that $\|F\|_p^p=1$. By the weak convergence and (\ref{dep}), we may assume that $\|F\|_p^p=a\in (0,1]$ and $\lim \|F^{\Phi}_{i_m}-F\|_p^p=1-a$. The fact that $F$ is a minimizer implies equalities in (\ref{str}) in limit; i.e.
\begin{displaymath}
    1-a^{\frac{2}{p}}=(1-a)^{\frac{2}{p}}=1-a^{\frac{2}{p}}.
\end{displaymath}
This gives the conclusion because $1<a^{\frac{2}{p}}+(1-a)^{\frac{2}{p}}$ for $a\in (0,1)$.

Moreover, by (\ref{infimum}), the strong convergence in $L^p(S^{2n+1})$ implies that $\int_{S^{2n+1}}\bar{F} \mathcal{A}_{2\gamma}^{\theta_0}(F) d\xi={C}^{-1}_{n,2\gamma}$. Combining with the fact that
\begin{equation}
    \begin{split}
        &\lim_{i_m\to \infty}\int_{S^{2n+1}}\left(\bar{F}^{\Phi}_{i_m}-\bar{F}\right) \mathcal{A}_{2\gamma}^{\theta_0} \left(F^{\Phi}_{i_m}-F\right) d\xi+\int_{S^{2n+1}}\bar{F} \mathcal{A}_{2\gamma}^{\theta_0}(F) d\xi\\
        =&\lim_{i_m\to \infty}\int_{S^{2n+1}}\bar{F}^{\Phi}_{i_m} \mathcal{A}_{2\gamma}^{\theta_0} (F^{\Phi}_{i_m}) d\xi={C}^{-1}_{n,2\gamma},
    \end{split}
\end{equation}
we conclude that 
\begin{equation}\label{norm}
    \lim_{i_m\to \infty}\int_{S^{2n+1}}\left(\bar{F}^{\Phi}_{i_m}-\bar{F}\right) \mathcal{A}_{2\gamma}^{\theta_0} \left(F^{\Phi}_{i_m}-F\right) d\xi=0.
\end{equation}
By the definition of the Folland--Stein space, (\ref{norm}) is equivalent to $\|F^{\Phi}_{i_m}-F\|_{S^{\gamma,2}}=0$, as $i_m\to \infty$, which means the strong convergence of $\{F^{\Phi}_{i_m}\}$ in $S^{\gamma,2}(S^{2n+1})$. 
\end{proof}
\begin{corollary}\label{transl}
Let $(S^{2n+1}, T^{1,0}S^{2n+1}, \theta_0)$ be the CR unit sphere with the volume element $d\xi$. Suppose that $u$ is a positive local minimizer of \begin{displaymath}
Y_{\gamma}(S^{2n+1})=\inf \left\{A^{\theta_0}_{\gamma}(F);F\in S^{\gamma,2}(S^{2n+1}), B^{\theta_0}_{\gamma}(F)=1 \right\}, \quad 0<\gamma<n+1,
\end{displaymath}
where $A^{\theta_0}_{\gamma}(F)$ and $B^{\theta_0}_{\gamma}(F)$ are defined by 
\begin{align*}
    A^{\theta_0}_{\gamma}(F)& :=\int_{S^{2n+1}} \bar{F} \mathcal{A}_{2\gamma}^{\theta_0}(F) d\xi, \quad F\in S^{\gamma,2}(S^{2n+1}),\\
    B^{\theta_0}_{\gamma}(F)& :=\int_{S^{2n+1}} |F|^p d\xi, \quad p=\frac{2Q}{Q-2\gamma}.
\end{align*}
Then there is an element $\Phi$ of the CR automorphism group ${\mathcal{A}}(S^{2n+1})$ of $S^{2n+1}$ such that
\begin{equation}\label{replace}
    u^{\Phi}=|J_{\Phi}|^{\frac{1}{p}}\Phi^{*}u
\end{equation}
is a positive minimizer of $Y_{\gamma}(S^{2n+1})$ which satisfies (\ref{balanced}), where $|J_{\Phi}|$ is the determinant of the Jacobian of $\Phi$.
\end{corollary}
\begin{proof}
    Since $A^{\theta_0}_{\gamma}$ and $B^{\theta_0}_{\gamma}$ are CR covariant, $u$ is a local minimizer of $Y_{\gamma}(S^{2n+1})$ if and only if $u^{\Phi}$ is a local minimizer of $Y_{\gamma}(S^{2n+1})$ for each $\Phi\in {\mathcal{A}}{\textbf{ut}}(S^{2n+1})$. Since $u$ is positive, $B^{\theta_0}_{\gamma}(u)\neq 0$. It follows from \cite[Lemma B.1]{FL10} that there is a $\Phi\in {\mathcal{A}}{\textbf{ut}}(S^{2n+1})$ such that $u^{\Phi}$ satisfies (\ref{balanced}).
\end{proof}

Next, we classify optimizers of sharp Sobolev inequalities (\ref{sharpSo}) following Frank--Lieb argument developed in \cite{Case19}. As explained in the introduction, the proof of Theorem \ref{class} consists of three ingredients. The desired spectral estimate are given in the following result.

\begin{proposition}\label{spectral}
Let $(S^{2n+1}, T^{1,0}S^{2n+1}, \theta_0)$ be the CR unit sphere with the volume element $d\xi$. Suppose that $u$ is a positive local minimizer of $Y_{\gamma}(S^{2n+1})$. Suppose additionally that 
\begin{equation}\label{cbalanced}
    \int_{S^{2n+1}}z_l u^p d\xi=0, \quad l=1,\cdots, n+1,
\end{equation}
where $z_1,\cdots, z_{n+1}$ are coordinates on $\C^{n+1}$ and we regard $S^{2n+1}\subset \C^{n+1}$ as the unit sphere. Then
\begin{equation}\label{spectral3}
    \sum_{l=1}^{n+1}\int_{S^{2n+1}}\bar{z}_l u[\mathcal{A}_{2\gamma}^{\theta_0},z_l](u)d\xi \geqslant (p-2) \int_{S^{2n+1}}u\mathcal{A}_{2\gamma}^{\theta_0}(u)d\xi,
\end{equation}
where 
\begin{displaymath}
    [\mathcal{A}_{2\gamma}^{\theta_0},z_l](u):=\mathcal{A}_{2\gamma}^{\theta_0}(z_l u)-z_l\mathcal{A}_{2\gamma}^{\theta_0}(u).
\end{displaymath}
\end{proposition}

We first prove the commutator identity involving the intertwining operators needed to execute the Frank--Lieb argument.

\begin{proof}[Proof of Theorem \ref{commutator}]
For $Y_{j,k}\in \mathcal{H}^{\S}_{j,k}$, we denote $\mathcal{Y}_{j,k}\in \mathcal{H}_{j,k}$ the homogeneous extension of $Y_{j,k}$ in $\C^{n+1}$. 

It follows from the algorithm in \cite{AR95} that for each coordinate function $z_l$ on $\C^{n+1}$, we have the decomposition
\begin{equation}
    z_l\mathcal{Y}_{j,k}=z_l\mathcal{Y}_{j,k}-\frac{|z|^2}{n+j+k}\frac{\partial}{\partial \bar{z}_l}\mathcal{Y}_{j,k}+\frac{|z|^2}{n+j+k}\frac{\partial}{\partial \bar{z}_l}\mathcal{Y}_{j,k},
\end{equation}
where
\begin{equation*}
    z_l\mathcal{Y}_{j,k}-\frac{|z|^2}{n+j+k}\frac{\partial}{\partial \bar{z}_l}\mathcal{Y}_{j,k}\in \mathcal{H}_{j+1,k}, \quad \quad \frac{\partial}{\partial \bar{z}_l}\mathcal{Y}_{j,k}\in \mathcal{H}_{j,k-1}.
\end{equation*}
Coming this with (\ref{form}) yields
\begin{equation}\label{cal1}
\begin{split}
    \mathcal{A}_{w,w'}^{\theta_0}\left(z_lY_{j,k}\right)&=\lambda_{j+1}(w)\lambda_{k}(w')\left(z_l\mathcal{Y}_{j,k}-\frac{|z|^2}{n+j+k}\frac{\partial}{\partial \bar{z}_l}\mathcal{Y}_{j,k}\right)\bigg|_{S^{2n+1}}\\
    &+\frac{\lambda_{j}(w)\lambda_{k-1}(w')}{n+j+k}\left(|z|^2\frac{\partial}{\partial \bar{z}_l}\mathcal{Y}_{j,k}\right)\bigg|_{S^{2n+1}}.
\end{split}
\end{equation}
Applying (\ref{cal1}) to each $z_l$, $l=1,\cdots, n+1$, multiplying $\bar{z}_l$ and taking the sum, we obtain
\begin{equation}\label{cal2}
\begin{split}
     &\left(\sum_{l=1}^{n+1}\bar{z}_l \left[\mathcal{A}_{w,w'}^{\theta_0},z_l\right]\right)Y_{j,k}=\left(\lambda_{j+1}(w)\lambda_{k}(w')-\lambda_{j}(w)\lambda_{k}(w')\right)Y_{j,k}\\
     &+\frac{\lambda_{j}(w)\lambda_{k-1}(w')-\lambda_{j+1}(w)\lambda_{k}(w')}{n+j+k} \left(\sum_{l=1}^{n+1}\bar{z}_l\frac{\partial}{\partial \bar{z}_l}\mathcal{Y}_{j,k}\right)\bigg|_{S^{2n+1}}.
     \end{split}
\end{equation}
By the explicit value of spectrum (\ref{spectrum}), we have
\begin{equation}\label{relation1}
\begin{split}
    \lambda_{j+1}(w)\lambda_{k}(w')&=\frac{j+\gamma-w}{j-w}\lambda_{j}(w)\lambda_{k}(w'),\\
    \lambda_{j}(w)\lambda_{k-1}(w')&=\frac{k-1-w'}{k-1+\gamma-w'}\lambda_{j}(w)\lambda_{k}(w').
    \end{split}
\end{equation}
Moreover, notice that 
\begin{equation}\label{relation2}
    \left(\sum_{l=1}^{n+1}\bar{z}_l\frac{\partial}{\partial \bar{z}_l}\mathcal{Y}_{j,k}\right)\bigg|_{S^{2n+1}}=\overline{Z}\mathcal{Y}_{j,k}\bigg|_{S^{2n+1}}=kY_{j,k}.
\end{equation}
Plugging (\ref{relation1}) and (\ref{relation2}) into (\ref{cal2}) yields
\begin{equation}\label{same}
    \begin{split}
     \left(\sum_{l=1}^{n+1}\bar{z}_l \left[\mathcal{A}_{w,w'}^{\theta_0},z_l\right]\right)Y_{j,k}&=\frac{\gamma (\gamma-1-w')}{\left(j-w\right)\left(k-1+\gamma-w'\right)}\lambda_{j}(w)\lambda_{k}(w')Y_{j,k}\\
     &=\gamma (\gamma-1-w')\cdot \frac{\lambda_j(w)}{j-w}\cdot\frac{\lambda_{k}(w')}{k-1+\gamma-w'} Y_{j,k}\\
     &=\gamma (\gamma-1-w') \lambda_j(w-1)\lambda_k(w') Y_{j,k}\\
     &=\gamma (\gamma-1-w') \mathcal{A}_{w-1,w'}^{\theta_0}Y_{j,k}.
     \end{split}
\end{equation}
By Proposition \ref{generalclass}, the desired result follows from (\ref{same}).
\end{proof}

We now turn to the desired spectral estimate. The fact that $\mathcal{A}_{2\gamma}^{\theta_0}$ is formally self-adjoint implies that if $u_t$ is a one-parameter family of functions in $S^{\gamma,2}(S^{2n+1}; \R)$ with $u_0=u$, then
\begin{align}
   \label{variation1}\frac{d}{dt}\bigg|_{t=0}A^{\theta_0}_{\gamma}(u_t)&=2\int_{S^{2n+1}}\dot{u}\mathcal{A}_{2\gamma}^{\theta_0} (u) d\xi, \\ \label{variation2}
    \frac{d^2}{dt^2}\bigg|_{t=0}A^{\theta_0}_{\gamma}(u_t)&=2\int_{S^{2n+1}}\dot{u}\mathcal{A}_{2\gamma}^{\theta_0} (\dot{u}) d\xi+2\int_{S^{2n+1}}\ddot{u}\mathcal{A}_{2\gamma}^{\theta_0} (u) d\xi. 
\end{align}
for $\dot{u}:=\frac{\partial}{\partial t}\big|_{t=0}u_t$ and $\ddot{u}:=\frac{\partial^2}{\partial t^2}\big|_{t=0}u_t$.
\begin{proof}[Proof of Proposition \ref{spectral}]
Let $v\in C^{\infty}(S^{2n+1};\R)$ be such that
\begin{equation}\label{ccbalanced}
    \int_{S^{2n+1}}v u^{p-1}d\xi=0.
\end{equation}
Then $u_t=\|u+tv\|^{-1}_{p}(u+tv)$ defines a smooth curve with $B^{\theta_0}_\gamma(u_t)=1$ and $u_0=u$, $\dot{u}_t\big|_{t=0}=v$. Since $u$ is a critical point of $A^{\theta_0}_\gamma$, it follows from (\ref{variation1}) that
\begin{displaymath}
    P^{\theta_0}_\gamma(u)=A^{\theta_0}_\gamma(u)u^{p-1}.
\end{displaymath}
Since $u$ is a local minimizer, $\frac{d^2}{dt^2}\bigg|_{t=0}A^{\theta_0}_\gamma(u_t)\geqslant 0$. Expanding this using (\ref{variation2}) and the above display yields
\begin{equation}\label{spectral2}
    \int_{S^{2n+1}}v \mathcal{A}_{2\gamma}^{\theta_0}(v)d\xi \geqslant (p-1)A^{\theta_0}_\gamma(u)\int_{S^{2n+1}}u^{p-2}v^2d\xi.
\end{equation}
The assumption (\ref{cbalanced}) implies that for each coordinate function $z_l=x_l+i y_l$, $1\leqslant l\leqslant n+1$, the functions $x_l u$ and $y_l u$ satisfy (\ref{ccbalanced}). Since $\mathcal{A}_{2\gamma}^{\theta_0}$ is self-adjoint, it follows from (\ref{spectral2}) that
\begin{displaymath}
\begin{split}
    \sum_{l=1}^{n+1}\int_{S^{2n+1}}\bar{z}_l u\mathcal{A}_{2\gamma}^{\theta_0}(z_lu)d\xi&=\sum_{l=1}^{n+1}\int_{S^{2n+1}}\left(x_l u\mathcal{A}_{2\gamma}^{\theta_0}(x_lu)+y_l u\mathcal{A}_{2\gamma}^{\theta_0}(y_lu)\right)d\xi\\
    &\geqslant (p-1)A^{\theta_0}_\gamma(u)\int_{S^{2n+1}}\sum_{l=1}^{n+1}(x_l^2+y_l^2)u^{p}d\xi\\
    &=(p-1)A^{\theta_0}_\gamma(u).
    \end{split}
\end{displaymath}
The final conclusion follows from the definition of $[\mathcal{A}_{2\gamma}^{\theta_0},z_j]$.
\end{proof}

Proposition \ref{spectral} and Corollary \ref{transl} reduce the problem of classifying positive local minimizers of $A^{\theta_0}_\gamma$ to the problem of showing that the only functions which satisfy (\ref{spectral3}) are the constants. This can be done by using the commutator identity in Theorem \ref{commutator}.

\begin{proof}[Proof of Theorem \ref{class}]
All computations in this proof are carried out with respect to the CR unit sphere $(S^{2n+1}, T^{1,0}S^{2n+1}, \theta_0)$ with the volume element $d\xi$. As discussed in \cite[Theorem 1.3]{FL10}, we may assume that the optimizer $u$ of (\ref{sharpSo}) is a positive real function; i.e. $u$ is a positive real minimizer of 
\begin{displaymath}
A^{\theta_0}_\gamma(u)=\int_{S^{2n+1}}uP^{\theta_0}_\gamma(u)d\xi
\end{displaymath}
with $B^{\theta_0}_\gamma(u)=1$. By Corollary \ref{transl} we may assume that $u$ satisfies (\ref{cbalanced}). We conclude from Proposition \ref{spectral} that
\begin{equation}\label{spe}
\sum_{l=l}^{n+1}\int_{S^{2n+1}}\bar{z}_l u[\mathcal{A}_{2\gamma}^{\theta_0},z_l](u)d\xi \geqslant (p-2) \int_{S^{2n+1}}u\mathcal{A}_{2\gamma}^{\theta_0}(u)d\xi.
\end{equation}
Combining this with Theorem \ref{commutator} yields 
\begin{equation}\label{linear}
    0\geqslant \int_{S^{2n+1}} u\left((p-2)\mathcal{A}_{2\gamma}^{\theta_0}-\gamma (\gamma-1-w) \mathcal{A}_{w-1,w}^{\theta_0}\right)(u)d\xi, \quad w=\frac{\gamma-n-1}{2}.
\end{equation}
Direct computation shows that for $Y_{j,k}\in \mathcal{H}^{\S}_{j,k}$,
\begin{equation*}
\begin{split}
    &\int_{S^{2n+1}} \overline{Y}_{j,k}\left((p-2)\mathcal{A}_{2\gamma}^{\theta_0}-\gamma (\gamma-1-w)\mathcal{A}_{w-1,w}^{\theta_0}\right)\left(Y_{j,k}\right) d\xi\\
    &=\left(\frac{2\gamma}{n+1-\gamma}-\frac{\gamma \frac{n+\gamma-1}{2}}{(j+\frac{n+1-\gamma}{2})(k+\frac{n+\gamma-1}{2})}\right)\lambda_j(w)\lambda_{k}(w)\\
    &\geqslant \left(\frac{2\gamma}{n+1-\gamma}-\frac{\gamma \frac{n+\gamma-1}{2}}{(\frac{n+1-\gamma}{2})(\frac{n+\gamma-1}{2})}\right)\lambda_j(w)\lambda_{k}(w)=0,
    \end{split}
\end{equation*}
where equality holds if and only if $(j,k)=(0,0)$. We conclude that the operator in (\ref{linear}) is nonnegative with kernel exactly equal to the constant functions. Combing this with (\ref{spe}) yields $u$ is constant. The final conclusion follows from the fact that if $\Phi\in {\mathcal{A}}{\textbf{ut}}(S^{2n+1})$, then 
\begin{displaymath}
|J_{\Phi}|^Q(\eta)=\frac{C}{|1-\xi \cdot \bar{\eta}|^Q}, \quad \eta\in S^{2n+1}
\end{displaymath}
for some $C>0, \xi\in \C^{n+1}, |\xi|<1$.
\end{proof}

\section{Classical sharp Sobolev inequalities}
The main result in this section is the proof of Theorem \ref{classicalcom} and the classification of optimizers of the classical sharp Sobolev inequalities (\ref{classicalHLS}).

\begin{proof}[Proof of Theorem \ref{classicalcom}]
For $Y_{h}\in \mathcal{H}^{\S}_{h}$, we denote $\mathcal{Y}_{h}\in \mathcal{H}_{h}$ the homogeneous extension of $Y_{h}$ in $\R^{n+1}$.

It follows from the algorithm in \cite{AR95} that for each coordinate function $x_j$ on $\R^{n+1}$, we have the decomposition
\begin{equation}
    x_j\mathcal{Y}_{h}=x_j\mathcal{Y}_{h}+\frac{|x|^2}{1-n-2h}\frac{\partial}{\partial x_j}\mathcal{Y}_{h}-\frac{|x|^2}{1-n-2h}\frac{\partial}{\partial x_j}\mathcal{Y}_{h},
\end{equation}
where
\begin{equation*}
    x_j\mathcal{Y}_{h}+\frac{|x|^2}{1-n-2h}\frac{\partial}{\partial x_j}\mathcal{Y}_{h}\in \mathcal{H}_{h+1}, \quad \quad \frac{\partial}{\partial x_j}\mathcal{Y}_{h}\in \mathcal{H}_{h-1}.
\end{equation*}
Combining this with the fact that for $Y_{h}\in \mathcal{H}^{\S}_{h}$,
\begin{equation}\label{spectrum2}
    P_{2\gamma}^{g_c}Y_h=\mu_h(2\gamma)Y_h,\quad \quad \mu_h(2\gamma)=\frac{\Gamma(h+\frac{n}{2}+\gamma)}{\Gamma(h+\frac{n}{2}-\gamma)},
\end{equation}
we obtain
\begin{equation}\label{cal4}
\begin{split}
     P_{2\gamma}^{g_c}\left(x_jY_h\right)=&\mu_{h+1}(2\gamma)\left(x_j\mathcal{Y}_{h}+\frac{|x|^2}{1-n-2h}\frac{\partial}{\partial x_j}\mathcal{Y}_{h}\right)\bigg|_{S^n}\\
     &-\mu_{h-1}(2\gamma)\left(\frac{|x|^2}{1-n-2h}\frac{\partial}{\partial x_j}\mathcal{Y}_{h}\right)\bigg|_{S^n}.
     \end{split}
\end{equation}
Applying (\ref{cal4}) to each $x_j$, $j=1,\cdots, n+1$, multiplying $x_j$ and taking the sum, we have
\begin{equation}\label{cal3}
\begin{split}
\left(\sum_{j=1}^{n+1}x_j\left[P_{2\gamma}^{g_c},x_j\right]\right)Y_h&=\left(\mu_{h+1}(2\gamma)-\mu_{h}(2\gamma)\right)Y_h\\
&+\frac{\mu_{h+1}(2\gamma)-\mu_{h-1}(2\gamma)}{1-n-2h}\left(|x|^2\sum_{j=1}^{n+1}\frac{\partial}{\partial x_j}\mathcal{Y}_{h}\right)\bigg|_{S^n}.
\end{split}
\end{equation}
By the explicit value of $\mu_h(2\gamma)$ (\ref{spectrum2}), we have
\begin{equation}\label{cal5}
    \begin{split}
        \mu_{h-1}(2\gamma)&=\left(h+\frac{n}{2}-\gamma\right)\left(h+\frac{n}{2}-\gamma-1\right) \mu_{h}(2(\gamma-1)),\\
        \mu_{h}(2\gamma)&=\left(h+\frac{n}{2}-\gamma\right)\left(h+\frac{n}{2}+\gamma-1\right) \mu_{h}(2(\gamma-1)),\\
        \mu_{h+1}(2\gamma)&=\left(h+\frac{n}{2}+\gamma\right)\left(h+\frac{n}{2}+\gamma-1\right) \mu_{h}(2(\gamma-1)).
    \end{split}
\end{equation}
Moreover, notice that 
\begin{equation}\label{relation3}
    \left(\sum_{j=1}^{n+1}x_j\frac{\partial}{\partial x_j}\mathcal{Y}_{h}\right)\bigg|_{S^{n}}=hY_{h}.
\end{equation}
Plugging (\ref{cal4}) and (\ref{cal5}) into (\ref{cal3}) yields
\begin{equation}\label{spectrum2}
\left(\sum_{j=1}^{n+1}x_j\left[P_{2\gamma}^{g_c},x_j\right]\right)Y_h=\gamma(n+2\gamma-2)P_{2(\gamma-1)}^{g_c}Y_h.
\end{equation}
Let
\begin{displaymath}
R_{2\gamma}^{g_c}:=\sum_{j=1}^{n+1}x_j\left[P_{2\gamma}^{g_c},x_j\right]-\gamma(n+2\gamma-2)P_{2(\gamma-1)}^{g_c}.
\end{displaymath}
(\ref{spectrum2}) implies that $R_{2\gamma}^{g_c}=0$ on $\mathrm{Span}\left\{\bigoplus_{h\in \N}\mathcal{H}^{\S}_h\right\}$. Under the standard inner product $(\cdot,\cdot)$ in $L^2(S^n)$, we have that for any $H\in \mathrm{Span}\left\{\bigoplus_{h\in \N}\mathcal{H}^{\S}_h\right\}$, $F\in C^{\infty}(S^n)$, 
\begin{equation*}
    \left(R_{2\gamma}^{g_c}(F),H\right)=\left(F,R_{2\gamma}^{g_c}(H)\right)=0,
\end{equation*}
where we use the self-adjointness of $R_{2\gamma}^{g_c}$. By the continuity of the inner product and the fact that $L^2(S^n)=\overline{\mathrm{Span}\left\{\bigoplus_{h\in \N}\mathcal{H}^{\S}_h\right\}}$, we obtain that
\begin{equation*}
    \left(R_{2\gamma}^{g_c}(F),G\right)=0, \quad \forall F\in C^{\infty}(S^n), G\in L^2(S^n),
\end{equation*}
which implies that $R_{2\gamma}^{g_c}=0$ on $C^{\infty}(S^n)$.
\end{proof}

The following proposition is a direct consequence of Lemma 4.2 in \cite{Case19}.
\begin{proposition}\label{trans2}
Let $(S^n, g_c)$ be the standard sphere with the canonical metric $g_c$ and $P_{2\gamma}^{g_c}$ be the intertwining operator of order $2\gamma$. Suppose that $u$ is a positive local minimizer of \begin{displaymath}
Y_{\gamma}(S^{n})=\inf \left\{A^{g_c}_{\gamma}(F);F\in W^{\gamma,2}(S^{n}), B^{g_c}_{\gamma}(F)=1 \right\}, \quad 0<\gamma<\frac{n}{2},
\end{displaymath}
where $A^{g_c}_{\gamma}(F)$ and $B^{g_c}_{\gamma}(F)$ are defined by 
\begin{align*}
    A^{g_c}_{\gamma}(F)& :=\int_{S^{n}} F P_{2\gamma}^{g_c}(F) d\sigma, \quad F\in W^{\gamma,2}(S^{n}),\\
    B^{g_c}_{\gamma}(F)& :=\int_{S^{n}} |F|^p d\sigma, \quad p=\frac{2n}{n-2\gamma}.
\end{align*}
Then there is an element $\Phi$ of the conformal group ${\mathrm{Conf}}(S^{n})$ of $S^{n}$ such that
\begin{equation}\label{replace}
    u^{\Phi}=|J_{\Phi}|^{\frac{1}{p}}\Phi^{*}u
\end{equation}
is a positive minimizer of $Y_{\gamma}(S^{n})$ which satisfies \begin{equation}\label{ebalanced}
    \int_{S^{n}}x_l|F|^p d\sigma=0, \quad l=1,\cdots, n+1,
\end{equation} 
where $|J_{\Phi}|$ is the determinant of the Jacobian of $\Phi$.
\end{proposition}
By a similar argument in Proposition \ref{spectral}, we obtain the desired spectral estimate as follows.
\begin{proposition}
Let $(S^n, g_c)$ be the standard sphere with the canonical metric $g_c$ and $P_{2\gamma}^{g_c}$ be the intertwining operator of order $2\gamma$. Suppose that $u$ is a positive local minimizer of $Y_{\gamma}(S^{n})$. Suppose additionally that 
\begin{equation}
    \int_{S^{n}}x_l u^p d\xi=0, \quad l=1,\cdots, n+1,
\end{equation}
where $x_1,\cdots, x_{n+1}$ are coordinates on $\R^{n+1}$. Then
\begin{equation}\label{spectral4}
    \sum_{l=1}^{n+1}\int_{S^{n}}x_l u[P_{2\gamma}^{g_c},x_l](u)d\sigma \geqslant (p-2) \int_{S^{n}}uP^{g_c}_{2\gamma}(u)d\sigma,
\end{equation}
where 
\begin{displaymath}
    [P_{2\gamma}^{g_c},x_l](u):=P_{2\gamma}^{g_c}(x_l u)-x_lP_{2\gamma}^{g_c}(u).
\end{displaymath}
\end{proposition}
Now, we can give a new proof of classical sharp Sobolev inequalities (\ref{classicalHLS}).

\begin{proof}[A new proof of Theorem \ref{classicalHLS}]
All computations in this proof are carried out with respect to the standard unit sphere $(S^{n}, g_c)$ with the volume element $d\sigma$. As discussed in \cite[Theorem 3.1]{FL12a}, we may assume that the optimizer $u$ of (\ref{classicalHLS}) is a positive real function; i.e. $u$ is a positive real minimizer of 
\begin{displaymath}
A^{g_c}_\gamma(u)=\int_{S^{n}}uP^{g_c}_2{\gamma}(u)d\sigma
\end{displaymath}
with $B^{g_c}_\gamma(u)=1$. By Corollary \ref{trans2} we may assume that $u$ satisfies (\ref{ebalanced}). We conclude from Proposition \ref{spectral2} that
\begin{equation}\label{spe2}
\sum_{l=1}^{n+1}\int_{S^{n}}x_l u[P_{2\gamma}^{g_c},x_l](u)d\sigma \geqslant (p-2) \int_{S^{n}}uP^{g_c}_{2\gamma}(u)d\sigma.
\end{equation}
Combining this with Theorem \ref{classicalcom} yields
\begin{equation}
    0\geqslant \int_{S^n}u\left((p-2)P^{g_c}_{2\gamma}-\gamma\left(n+2\gamma-2\right)P^{g_c}_{2(\gamma-1)}\right)(u)d\sigma.
\end{equation}
It follows from the explicit form (\ref{explicitform}) of $P^{g_c}_\gamma$ that 
\begin{equation}
    P^{g_c}_{2\gamma}=\left(-\Delta_{S^n}+\frac{(n-2\gamma)(n+2\gamma-2)}{4}\right)P^{g_c}_{2(\gamma-1)}.
\end{equation}
Hence, by simple calculation, we find that
\begin{equation}\label{expansion2}
    (p-2)P^{g_c}_{2\gamma}-\gamma\left(n+2\gamma-2\right)P^{g_c}_{2(\gamma-1)}=\frac{4\gamma}{n-2\gamma}\left(-\Delta_{S^n}\right)P^{g_c}_{2(\gamma-1)}.
\end{equation}
We conclude that the operator in (\ref{expansion2}) is nonnegative with kernel exactly equal to the constant functions. Combing this with (\ref{spe2}) yields $u$ is constant. The final conclusion follows from the fact that if $\Phi\in {\mathrm{Conf}}(S^{n})$, then the Jacobian determinant of $\Phi$ is the form of
\begin{displaymath}
\frac{C}{|1-\langle\eta, \xi\rangle|^n}, \quad \eta\in S^{n}
\end{displaymath}
for some $C>0, \xi\in \R^{n+1}, |\xi|<1$.
\end{proof}

\appendix
\section{Intertwining operators on $S^{2n+1}$}
Following the idea in \cite{BFM07}, in this appendix we give an explicit calculation of the spectrum of the intertwining operators $\mathcal{A}_{w,w'}^{\theta_0}$ as defined by (\ref{intertwining}); a consequence of this calculation will be formula (\ref{spectrum}) up to a constant.

\begin{proposition}\label{generalclass}
    For $0<\gamma<\frac{Q}{2}$, let $(w,w')\in \R^2$ such that $w+w'+n+1=\gamma$. Suppose that the operator $\mathcal{A}_{w,w'}^{\theta_0}$ is formally self-adjoint in $S^{\gamma,2}$ and is intertwining, i.e., for any $F\in C^{\infty}(S^{2n+1})$, $\tau\in \mathcal{A}{\textbf{ut}}(S^{2n+1})$,
\begin{equation}
    J_{\tau}^{\gamma-w}\bar{J}_{\tau}^{\gamma-w'}\left(\mathcal{A}_{w,w'}^{\theta_0}F\right)\circ \tau= \mathcal{A}_{w,w'}^{\theta_0}\left(J_{\tau}^{-w}\bar{J}_{\tau}^{-w'}\left(F\circ \tau\right)\right).
\end{equation}
Then $\mathcal{A}_{w,w'}^{\theta_0}$ is diagonal with respect to the spherical harmonics, and for every $Y_{j,k}\in \mathcal{H}^{\S}_{j,k}$,
\begin{equation*}
    \mathcal{A}_{w,w'}^{\theta_0}Y_{j,k}=c\lambda_j(w)\lambda_k(w')Y_{j,k}
\end{equation*}
for some constant $c\in \R$, with
\begin{equation*}
    \lambda_j(w)=\frac{\Gamma(j+\gamma-w)}{\Gamma(j-w)}, \quad \lambda_k(w')=\frac{\Gamma(k+\gamma-w')}{\Gamma(k-w')}.
\end{equation*}
Vice verse, a self-adjoint operator $\mathcal{Q}_{w,w'}^{\theta_0}$ with eigenvalues $\lambda_j(w)\lambda_k(w')$ is intertwining. 
\end{proposition}

\begin{proof}
    The fact that $\mathcal{A}_{w,w'}^{\theta_0}$ is diagonal follows from Schur's lemma and the irreducibility of the spaces $\mathcal{H}^{\S}_{j,k}$. Suppose that $\mathcal{A}_{w,w'}^{\theta_0}\Phi_{j,k}=\lambda_{j,k}\Phi_{j,k}$. From now on, in the formulation (\ref{zonal}) of $\Phi_{j,k}$, we choose $\eta$ to be the north pole of $S^{2n+1}$ and denote
    \begin{equation*}
        \Psi_{j,k}(z)=\Phi_{j,k}(\zeta,\eta)=\bar{z}^{k-j}P_j^{(n-1,k-j)}(2|z|^2-1), \quad z=\zeta_{n+1}, \quad j\leqslant k
    \end{equation*}
    such that we have $\mathcal{A}_{w,w'}^{\theta_0}\Psi_{j,k}=\lambda_{j,k}\Psi_{j,k}$.

    Consider the family of dilations on $\H^n$, which on $S^{2n+1}$ take the form
    \begin{equation}
        \tau_{\delta}(\zeta)=\tau_{\delta}(\zeta',\zeta_{n+1})\left(\frac{2\delta \zeta'}{1+\zeta_{n+1}+\delta^2(1-\zeta_{n+1})},\frac{1+\zeta_{n+1}-\delta^2(1-\zeta_{n+1})}{1+\zeta_{n+1}+\delta^2(1-\zeta_{n+1})}\right).
    \end{equation}
    The conformal factor $J_{\tau_{\delta}}$ of $\tau_{\delta}$ is given by
    \begin{equation*}
        J_{\tau_{\delta}}(\zeta)=\frac{2\delta}{1+\zeta_{n+1}+\delta^2(1-\zeta_{n+1})}.
    \end{equation*}
    Direct computation shows that
    \begin{equation*}
    \begin{split}
        \frac{d}{d\delta}\bigg|_{\delta=1}J_{\tau_{\delta}}^{-w}\bar{J}_{\tau_{\delta}}^{-w'}&=-wz-w'\bar{z},\\
        \frac{d}{d\delta}\bigg|_{\delta=1}\tau_{\delta}(\zeta)\cdot \eta&=z^2-1.
        \end{split}
    \end{equation*}
    Therefore, we have
    \begin{equation}\label{derivativezonal}
    \begin{split}
        \frac{d}{d\delta}\bigg|_{\delta=1} J_{\tau_{\delta}}^{-w}\bar{J}_{\tau_{\delta}}^{-w'}&\left(\Psi_{j,k}\circ \tau_{\delta}\right)=\left(-wz-w'\bar{z}\right)\bar{z}^{k-j}P_j^{(n-1,k-j)}(2|z|^2-1)\\
        &+(k-j)(\bar{z}^2-1)\bar{z}^{k-j-1}P_j^{(n-1,k-j)}(2|z|^2-1)\\
        &+2(z+\bar{z})(|z|^2-1)\bar{z}^{k-j}\frac{d}{dx}P_j^{(n-1,k-j)}(2|z|^2-1).
        \end{split}
    \end{equation}
    The above quantity is a polynomial in $z$ and $\bar{z}$ with highest order monomials that are multiples of $z^j\bar{z}^{k+1}$ and $z^{j+1}\bar{z}^k$. The projection of (\ref{derivativezonal}) on $\mathcal{H}^{\S}_{j+1,k}\bigoplus \mathcal{H}^{\S}_{j,k+1}$ gives that for fixed $0\leqslant j<k$,
    \begin{equation}\label{small}
        \begin{split}
            \frac{d}{d\delta}\bigg|_{\delta=1} J_{\tau_{\delta}}^{-w}\bar{J}_{\tau_{\delta}}^{-w'}\left(\Psi_{j,k}\circ \tau_{\delta}\right)\bigg|_{\mathcal{H}^{\S}_{j+1,k}\bigoplus \mathcal{H}^{\S}_{j,k+1}}&=A\bar{z}^{k-j-1}P_{j+1}^{(n-1,k-j-1)}(2|z|^2-1)\\
            &+B\bar{z}^{k-j+1}P_j^{(n-1,k-j+1)}(2|z|^2-1),
        \end{split}
    \end{equation}
    and for $j=k$,
    \begin{equation}\label{equal}
        \begin{split}
            \frac{d}{d\delta}\bigg|_{\delta=1} J_{\tau_{\delta}}^{-w}\bar{J}_{\tau_{\delta}}^{-w'}\left(\Psi_{j,k}\circ \tau_{\delta}\right)\bigg|_{\mathcal{H}^{\S}_{j+1,j}\bigoplus \mathcal{H}^{\S}_{j,j+1}}&=AzP_{j}^{(n-1,1)}(2|z|^2-1)\\
            &+B\bar{z}^{1}P_j^{(n-1,1)}(2|z|^2-1).
        \end{split}
    \end{equation}
    Our main goal is to determine $A$ and $B$. In order to do this, we consider the case $z$ real and $z$ imaginary, and we compare the coefficients of leading terms in (\ref{derivativezonal}) and (\ref{small})-(\ref{equal}). What we need here is the coefficient of $x^j$ in a Jacobi polynomials of order $j$ which is given by
    \begin{equation*}
        \frac{1}{j!}\frac{d^j}{dx^j}P_j^{(\alpha,\beta)}(x)=\frac{1}{2^jj!}\frac{\Gamma(2j+\alpha+\beta+1)}{\Gamma(j+\alpha+\beta+1)}.
    \end{equation*}
    On one hand, if $z$ is real, a comparison of the coefficients of $z^{k+j+1}$ yields
    \begin{equation}
        (-w-w'+k+j)\frac{\Gamma(k+j+n)}{j!\Gamma(k+n)}=A\frac{\Gamma(k+j+n+1)}{(j+1)!\Gamma(k+n)}+B\frac{\Gamma(k+j+n+1)}{j!\Gamma(k+n+1)},
    \end{equation}
    i.e.,
    \begin{equation}\label{real}
        (-w-w'+k+j)=A\frac{k+j+n}{j+1}+B\frac{k+j+n}{k+n}.
    \end{equation}
    On the other hand, if $z$ is purely imaginary, the same comparison yields
    \begin{equation}
    \begin{split}
        (-i)^{k-j+1}(k-j+w-w')\frac{\Gamma(k+j+n)}{j!\Gamma(k+n)}&=(-i)^{k-j-1}A\frac{\Gamma(k+j+n+1)}{(j+1)!\Gamma(k+n)}\\
        &+(-i)^{k-j+1}B\frac{\Gamma(k+j+n+1)}{j!\Gamma(k+n+1)},
        \end{split}
    \end{equation}
    i.e.,
    \begin{equation}\label{imaginary}
        k-j+w-w'=-A\frac{k+j+n}{j+1}+B\frac{k+j+n}{k+n}.
    \end{equation}
    Solving (\ref{real}) and (\ref{imaginary}) for $A$ and $B$, we obtain that
    \begin{equation}
        A=(j-w)\frac{j+1}{k+j+n},\quad   B=(k-w')\frac{k+n}{k+j+n}, 
    \end{equation}
    which means that for $0\leqslant j\leqslant k$,
    \begin{equation}\label{zonaldecompositon}
    \begin{split}
        &\frac{d}{d\delta}\bigg|_{\delta=1} J_{\tau_{\delta}}^{-w}\bar{J}_{\tau_{\delta}}^{-w'}\left(\Psi_{j,k}\circ \tau_{\delta}\right)\bigg|_{\mathcal{H}^{\S}_{j+1,k}\bigoplus \mathcal{H}^{\S}_{j,k+1}}=\\
&(j-w)\frac{j+1}{k+j+n}\Psi_{j+1,k}+(k-w')\frac{k+n}{k+j+n}\Psi_{j,k+1}       \end{split}
    \end{equation}
    The intertwining property of $\mathcal{A}_{w,w'}^{\theta_0}$ means that
    \begin{equation}
        \lambda_{j,k} J_{\tau_{\delta}}^{\gamma-w}\bar{J}_{\tau_{\delta}}^{\gamma-w'}\Psi_{j,k}\circ \tau_{\delta}= \mathcal{A}_{w,w'}^{\theta_0}\left(J_{\tau_{\delta}}^{-w}\bar{J}_{\tau_{\delta}}^{-w'}\left(\Psi_{j,k}\circ \tau_{\delta}\right)\right)
    \end{equation}
    Differentiating in $\delta$ and using (\ref{zonaldecompositon}) yield
    \begin{equation}\label{induction}
        \begin{split}
            &\lambda_{j,k}(j+\gamma-w)\frac{j+1}{k+j+n}\Psi_{j+1,k}+\lambda_{j,k}(k+\gamma-w')\frac{k+n}{k+j+n}\Psi_{j,k+1}\\
            &=\lambda_{j+1,k}(j-w)\frac{j+1}{k+j+n}\Psi_{j+1,k}+\lambda_{j,k+1}(k-w')\frac{k+n}{k+j+n}\Psi_{j,k+1}.  
        \end{split}
    \end{equation}
    which implies that
    \begin{equation}
        \lambda_{j+1,k}=\lambda_{j,k}\frac{j+\gamma-w}{j-w},\quad \lambda_{j,k+1}=\lambda_{j,k}\frac{k+\gamma-w'}{k-w'}, \quad 0\leqslant j\leqslant k.
    \end{equation}
    We conclude that for $0\leqslant j\leqslant k$, 
    \begin{equation}
       \lambda_{j,k}=\lambda_{0,k}\frac{\Gamma(j+\gamma-w)}{\Gamma(j-w)}=\lambda_{0,0}\frac{\Gamma(j+\gamma-w)}{\Gamma(j-w)}\frac{\Gamma(k+\gamma-w')}{\Gamma(k-w')}.
    \end{equation}
   Since $\Phi_{j,k}(\zeta,\eta)=\overline{\Phi_{k,j}(\zeta,\eta)}$, $k\leqslant j$ and the spectrum of $\mathcal{A}_{w,w'}^{\theta_0}$ is real, taking the conjugation in (\ref{induction}) yields that for $j,k\in \N_0$,
   \begin{equation*}
        \lambda_{j,k}=\lambda_{0,0}\frac{\Gamma(j+\gamma-w)}{\Gamma(j-w)}\frac{\Gamma(k+\gamma-w')}{\Gamma(k-w')}.
   \end{equation*}
Let $\lambda_{0,0}=1$ and $R_{w,w'}^{\theta_0}=\mathcal{A}_{w,w'}^{\theta_0}-\mathcal{Q}_{w,w'}^{\theta_0}$.
The assumption implies that $R_{w,w'}^{\theta_0}=0$ on $\mathrm{Span}\left\{\bigoplus_{j,k\in \N_0}\mathcal{H}^{\S}_{j,k}\right\}$. Under the standard inner product $(\cdot,\cdot)$ in $L^2(S^{2n+1})$, we have that for any $H\in \mathrm{Span}\left\{\bigoplus_{j,k\in \N_0}\mathcal{H}^{\S}_{j,k}\right\}$, $F\in C^{\infty}(S^{2n+1})$, 
\begin{equation*}
    \left(R_{w,w'}^{\theta_0}(F),H\right)=\left(F,R_{w,w'}^{\theta_0}(H)\right)=0,
\end{equation*}
where we use the self-adjointness of $R_{w,w'}^{\theta_0}$. By the continuity of the inner product and the fact that $L^2(S^{2n+1})=\overline{\mathrm{Span}\left\{\bigoplus_{j,k\in \N_0}\mathcal{H}^{\S}_{j,k}\right\}}$, we obtain that
\begin{equation*}
    \left(R_{w,w'}^{\theta_0}(F),G\right)=0, \quad \forall F\in C^{\infty}(S^{2n+1}), G\in L^2(S^{2n+1}),
\end{equation*}
which implies that $R_{w,w'}^{\theta_0}=0$ on $C^{\infty}(S^{2n+1})$.
   
\end{proof}

\bibliography{mybib}{}
\bibliographystyle{alpha}

\end{document}